\documentclass[11pt, a4paper]{article}
\usepackage{times}
\usepackage{a4wide}
\usepackage[british]{babel}
\usepackage{enumerate, longtable}
\usepackage{amsmath, amscd, amsfonts, amsthm, amssymb, latexsym, comment, stmaryrd, graphicx}
\usepackage[all]{xy}
\usepackage[T1]{fontenc}
\usepackage[utf8]{inputenc}
\usepackage{authblk}
\usepackage[usenames,dvipsnames]{color}

\newtheorem{thm}{Theorem}[section]
\newtheorem{lem}[thm]{Lemma}

\newtheorem{rem}[thm]{Remark}

\newtheorem{prop}[thm]{Proposition}
\newtheorem{cor}[thm]{Corollary}

\newcommand{\GL}{\mathrm{GL}}

\newcommand{\SL}{\mathrm{SL}}

\DeclareMathOperator{\Gal}{Gal}

\definecolor{MyDarkBlue}{rgb}{0,0.08,0.45}
\definecolor{MyDarkBlue}{rgb}{0,0,0}

\begin{document}

\selectlanguage{british}

\title{Locally cyclic extensions with Galois group $\GL_2(p)$}
\author[1]{Sara Arias-de-Reyna\thanks{sara$\_$arias@us.es}}
\affil[1]{Departamento de \'Algebra, Universidad de Sevilla, Spain}
\author[2]{Joachim K\"onig\thanks{jkoenig@knue.ac.kr}}
\affil[2]{Korea National University of Education, Cheongju, South Korea}
\date{}
\maketitle

\begin{abstract}
Using Galois representations attached to elliptic curves, we construct Galois extensions of $\mathbb{Q}$ with group $GL_2(p)$ in which all decomposition groups are cyclic. This is the first such realization for all primes $p$.
\end{abstract}

\section{Introduction and main results}
%
% Some first tries at ``main results" and reasoning behind them
%
Let $K/k$ be a Galois extension of number fields. We say that $K/k$ is \textit{locally cyclic}, if its decomposition group at every prime is cyclic. 

Locally cyclic extensions feature prominently in several areas of algebraic number theory.
 For example, they are useful for the solution of embedding problems in Galois theory. Recall that an embedding problem is the question whether, given a continuous epimorphism $\varphi: G_K\to G$ (corresponding to a Galois extension $L/K$ of Galois group $G$), and a short exact sequence of groups $1\to N\to \Gamma\stackrel{\pi}{\to} G\to 1$, one can find an epimorphism $\psi:G_k\to \Gamma$ such that $\pi\circ \psi = \varphi$ (thus embedding $L/K$ into a Galois extension of $K$ with group $\Gamma$). If, in this setting, $N$ is a central normal subgroup of $\Gamma$, then there is a  local-global principle for such an embedding problem, tying the solvability directly to the local behavior of the primes ramifying in $L/K$, and thus making it desirable to have ``easy" (such as, cyclic) decomposition groups. See \cite[\S 5]{Neukirch1973} for concrete results in this direction.
 On a related note, Shafarevich's famous realization of all solvable groups as Galois groups over $\mathbb{Q}$ makes use of specific kinds of locally cyclic extensions to solve certain \textit{split} embedding problems, and does in particular achieve existence of locally cyclic realizations for all solvable groups (see \cite[Theorem 9.6.7]{NSW}).

 For further applications, notably to weak approximation principle on norm-one tori, and to ``intersective polynomials", see \cite[\S 1]{KimKoenig}.

It is then natural to ask whether there exist (infinitely many) locally cyclic Galois extensions of $\mathbb{Q}$ with any given Galois group $G$.
The purpose of this paper is to provide a positive answer for all groups $GL_2(p)$ (and consequently, also for all the almost-simple groups $PGL_2(p)$).  Concretely, we show:

\begin{thm}
\label{thm:1}
Let $p\ge 5$ be a prime. Then there are infinitely many locally cyclic Galois extensions of $\mathbb{Q}$ with Galois group $GL_2(p)$. %Moreover, these extensions may be chosen to be pairwise linearly disjoint over $\mathbb{Q}(\sqrt{p^\star}) := \mathbb{Q}(\sqrt{(-1)^{(p-1)/2}p})$.
\end{thm}

In fact, we will show more strongly that the infinitely many extensions in Theorem \ref{thm:1} can be chosen pairwise linearly disjoint over $\mathbb{Q}(\sqrt{p^\star}) := \mathbb{Q}(\sqrt{(-1)^{(p-1)/2}p})$, see Corollary \ref{thm:disj}.

Locally cyclic extensions of $\mathbb{Q}$ have been long known to exist for solvable groups, due to Shafarevich's method.\footnote{In particular, Theorem \ref{thm:1} is already known for $p=3$, since $GL_2(3)$ is solvable.} In stark contrast to this, there were however, until recently, no examples at all of nonsolvable groups occurring as the Galois group of infinitely many locally cyclic Galois extensions of $\mathbb{Q}$. \footnote{At least, apart from trivial constructions such as using the same nonsolvable subextension and extending it by infinitely many distinct solvable ones.}
The first locally cyclic realizations were obtained in \cite{KimKoenig} for groups such as $G=S_5$ and $G=PGL_2(7)$, as well as certain central extensions of these. 

%{\bf Remark (JK):} These central extensions (last section of \cite{KimKoenig}) are such that they cover the case $GL_2(7)$, but not $GL_2(5)$.

In \cite{Koe21}, infinitely many of the groups $PGL_2(p)$ were realized as Galois groups of (infinitely many) locally cyclic extensions of $\mathbb{Q}$; however the majority (in a density sense when counting primes $p$) of those groups was still left open. This is, in part, due to the fact that previous approaches rested on specialization of $\mathbb{Q}$-regular Galois extensions of $\mathbb{Q}(t)$, an approach which has been fruitful for many problems in inverse Galois theory, but which also has certain limitations in practice.

Here, we approach the problem via constructing suitable elliptic curves and analyzing the local behavior of the $p$-division fields. This allows us to deal with all groups $GL_2(p)$ (and consequently also $PGL_2(p)$) simultaneously.
More precisely, we will consider semistable elliptic curves. The theory of Tate curves provides an explicit description of the action of the inertia group at the ramified primes on the $p$-torsion points of the elliptic curve. 

In particular, we will be interested in elliptic curves with multiplicative reduction at the prime $p$, and such that all the other primes of bad reduction are congruent to $1$ modulo $p$. To prove the existence of infinitely many such elliptic curves, we will make use of a deep result of Green and Tao on the representability of primes by linear equations.

The contents of this paper are organised as follows. First of all, we collect in Section \ref{sec:Tate_Curve} the necessary results concerning Tate curves. In Section \ref{sec:change_of_variables}, we perform a suitable change of variables on a Tate curve to obtain a Weierstrass equation with a convenient shape. The next section is devoted to writing explicit conditions modulo a suitable power of $p$ and $2$ that ensure that the field of $p$-division of $E$ has the locally cyclic property at these primes. Then we show how to apply the theorem of Green and Tao to construct the desired elliptic curves. In particular, the %existence part of Theorem \ref{thm:1} is obtained in 
assertion of Theorem \ref{thm:1} follows from Corollary \ref{cor:exist}. Finally, in Section \ref{sec:twisting} we show how to twist the solutions obtained in the previous section to ensure linear disjointness over $\mathbb{Q}(\sqrt{p}^\star)$.

\bigskip

\textbf{Acknowledgements}

S. Arias-de-Reyna was supported by projects US-1262169 (Junta de Andaluc\'ia and FEDER, UE) and PID2020-114613GB-I00 P (Ministerio de Ciencia e Innovaci\'on, Spain).
%\marginpar{Say some more about strategy of the different sections?}

%SA: I added a paragraph. Please make as many changes as you wish!

\section{The Tate Curve}\label{sec:Tate_Curve}

Let $p$ be a prime number. In this section we will recall some well-known facts about Tate curves (see e.g. \cite[Chapter V, Theorem 3.1]{AdvancedTopics}) over the field of $p$-adic numbers $\mathbb{Q}_p$, and expand on some results on the $p$-division fields of such curves and on explicit computations on the $2$-torsion points that will be useful for us in the next sections.

Denote by $\vert \cdot \vert_p$ the $p$-adic absolute value. To each $q\in \mathbb{Q}_p^{\times}$ satisfying $\vert q \vert_p<1$, we can attach an elliptic curve defined over $\mathbb{Q}_p$ in the following way. 

For each $k\in \mathbb{N}$, define the following quantities: 
\begin{equation}\label{eq:Tate} s_k(q)=\sum_{n\geq 1} \frac{n^kq^n}{1-q^n}, \quad a_4(q)=-s_3(q), \quad a_6(q)=-\frac{5s_3(q) + 7s_5(q)}{12}.\end{equation}

Note that the condition $\vert q \vert_p<1$ implies that the series $s_k(q)$ converge in $\mathbb{Q}_p$ for all $k\in \mathbb{N}$.

We further define the power series
\begin{equation}\label{eq:series}
\begin{cases}
 X(u, q)=\sum_{n\in \mathbb{Z}}\frac{q^nu}{(1- q^n u)^2} - 2s_1(q);\\
 Y(u, q)=\sum_{n\in \mathbb{Z}}\frac{(q^nu)^2}{(1 - q^nu)^3} + s_1(q).\\
 \end{cases}
\end{equation}

Note that these series converge for all $q$ with $\vert q \vert_p<1$ and all $u\in \overline{\mathbb{Q}}_p$, $u\not\in q^{\mathbb{Z}}=\{q^n: n\in \mathbb{Z}\}$.

\begin{thm}\label{thm:Tate_curve} Let $\vert q \vert_p<1$, and let $E_q$ be the plane projective curve defined over $\mathbb{Q}_p$ by the (affine) Weierstrass equation $y^2 + xy=x^3 + a_4(q) x + a_6(q)$.
Then the following hold:

\begin{enumerate}
 \item The curve $E_q$ is an elliptic curve. 
 
 \item The map 
\begin{equation}\label{eq:Tate-isomorphism}
\begin{aligned} \phi: \overline{\mathbb{Q}}_p^{\times} &\rightarrow E_q(\overline{\mathbb{Q}}_p)\\
 u & \mapsto\begin{cases}
             \begin{aligned}&(X(u, q), Y(u, q)) &\text{ if }u\not \in q^{\mathbb{Z}}\\
             &\mathcal{O} &\text{ if }u\in q^{\mathbb{Z}}\\\end{aligned}
            \end{cases}
\end{aligned}
\end{equation}
is a surjective group homomorphism with kernel $q^{\mathbb{Z}}$.
 
\item For all $\sigma\in \Gal(\overline{\mathbb{Q}}_p/\mathbb{Q}_p)$ and all $u\in \overline{\mathbb{Q}}_p$, 
\begin{equation*}
 \phi(\sigma(u))=\sigma(\phi(u)).
\end{equation*}

\end{enumerate}
 
\end{thm}

The previous theorem provides us with a group isomorphism (which we denote also by $\phi$ to simplify notation)
\begin{equation*}\phi: \overline{\mathbb{Q}}_p^{\times}/q^{\mathbb{Z}}\rightarrow E_q(\overline{\mathbb{Q}}_p)\end{equation*}
which is compatible with the Galois action. 

\subsection{$\ell$-torsion representation of $E_q$}\label{subsec:ell_torsion}

Let $\ell$ be a prime number (which can be equal to $p$ or different from $p$). The isomorphism above induces an isomorphism of $\mathrm{Gal}(\overline{\mathbb{Q}}_p/\mathbb{Q}_p)$-modules between the subgroup of $\ell$-torsion of $\overline{\mathbb{Q}}_p^{\times}/q^{\mathbb{Z}}$, which we will denote by $W_{\ell}$, and $E_q[\ell]$. 

Let us analyze the structure of $W_{\ell}$. For an element $u\in \overline{\mathbb{Q}}_p^{\times}$, denote by $[u]$ its class in $\overline{\mathbb{Q}}_p^{\times}/q^{\mathbb{Z}}$. We have the following result (cf. \cite[Appendix A.1.2, pag.~IV-31]{SerreAbelianl-adic}, \cite[1.12]{SerreProprietes}).

\begin{prop} We have an exact sequence of $\Gal(\overline{\mathbb{Q}}_p/\mathbb{Q}_p)$-modules, 
\begin{equation*}
 \xymatrix{0\ar[r]&\mu_{\ell}\ar[r]^{\psi_1}&  W_{\ell} \ar[r]^{\psi_2} & \mathbb{Z}/\ell\mathbb{Z} \ar[r]&  0,}
\end{equation*}
where $\mu_{\ell}$ denotes the group of $\ell$-th roots of unity and where the action of  $\Gal(\overline{\mathbb{Q}}_p/\mathbb{Q}_p)$ on $\mu_{\ell}$ occurs via the cyclotomic character and the action on $\mathbb{Z}/\ell\mathbb{Z}$ is the trivial action.
 \end{prop}

 In the proposition above, the map $\psi_1$ is just the composition of the inclusion $\mu_{\ell}\hookrightarrow \overline{\mathbb{Q}}_p^{\times}$ and the projection $\overline{\mathbb{Q}}_p^{\times}\rightarrow \overline{\mathbb{Q}}_p^{\times}/q^{\mathbb{Z}}$. The map $\psi_2$ can be defined as follows: for each $[u]\in  W_{\ell}$, we have that $[u]^{\ell}=[1]$, that is to say, there exists $n\in \mathbb{Z}$ such that $u^{\ell}=q^n$. Define $\psi_2([u])=n + \ell\mathbb{Z}\in \mathbb{Z}/\ell\mathbb{Z}$. It is easy to check that the map $\psi_2$ is a well-defined surjective group homomorphism. 
 
Choose an $\ell$-th root of $q$, say $q^{1/\ell}$, and a primitive $\ell$-th root of unity $\zeta_{\ell}$. Then the pair $\{\zeta_{\ell}, q^{1/\ell}\}$ is a basis of the $\mathbb{Z}$-module $W_{\ell}$. If we denote by $\rho_{\ell}:\Gal(\overline{\mathbb{Q}}_p/\mathbb{Q}_p)\rightarrow \GL(W_{\ell})$ the Galois representation induced by the Galois action on $W_{\ell}$, with respect to this basis we can write it in the shape
\begin{equation*}
\rho_{\ell}= \begin{pmatrix} \chi_{\ell} & * \\ 0 & 1\end{pmatrix}
\end{equation*}

If furthermore, $q^{1/\ell}\in \mathbb{Q}_p$, then $\rho_{\ell}$ can be written as
\begin{equation*}
 \rho_{\ell}=\begin{pmatrix} \chi_{\ell} & 0 \\ 0 & 1\end{pmatrix}
\end{equation*}
with respect to this basis. Therefore, its image is a cyclic group of order $\ell-1$.

In the conditions of Theorem \ref{thm:Tate_curve}, denote by $\rho_{E_q, \ell}$ the Galois representation attached to the $\ell$-torsion points of $E_q$. 
From the discussion above, we obtain the following two useful consequences:

\begin{cor}\label{cor:1modp}
 Let $\ell$ be a prime such that $p\equiv 1 \pmod{\ell}$. Then the image of $\Gal(\overline{\mathbb{Q}}_p/\mathbb{Q}_p)$ by $\rho_{E_q, \ell}$ is either trivial or a cyclic group of order $\ell$.
\end{cor}

\begin{proof}
 Note that $\chi_{\ell}(\mathrm{Frob}_p)\equiv p \equiv 1\pmod{\ell}$. 
\end{proof}

\begin{cor}\label{cor:cyclic_at_p} Let $p$ and $\ell$ be two primes (not necessarily different). Assume that $q^{1/\ell}\in \mathbb{Q}_p$. Then the image of $\Gal(\overline{\mathbb{Q}}_p/\mathbb{Q}_p)$ by $\rho_{E_q, \ell}$ is a cyclic group of order $\ell-1$.
\end{cor}

Finally, we have the following result, that ensures that any elliptic curve over $\mathbb{Q}_p$ whose $j$-invariant has negative $p$-adic valuation is isomorphic to a Tate curve (cf.~\cite[Lemmas 5.1 and 5.2 of Chapter V and the proof of Lemma 5.1]{AdvancedTopics}):

\begin{lem}\label{lem:isomorphic} Let $E$ be an elliptic curve defined over a finite extension of $\mathbb{Q}_p$ with $j$-invariant $\alpha\in \overline{\mathbb{Q}}_p$ satisfying that $\vert \alpha\vert_p>1$. Then there exists $q\in \mathbb{Q}_p(\alpha)$ with $\vert q \vert_p<1$ such that $E$ is isomorphic to $E_q$ over an extension of $\mathbb{Q}_p(\alpha)$ of degree less than or equal to $2$.
More precisely, there exists a power series $g(z)=z + \cdots \in \mathbb{Z}[[z]]$ such that $q=g(1/\alpha)$.
\end{lem}

We can be more precise about the relationship between the Galois representations attached to both elliptic curves (cf.~\cite[Lemmas 5.1 and 5.2  of Chapter V]{AdvancedTopics}). For simplicity we consider the case when both elliptic curves are defined over $\mathbb{Q}_p$.

\begin{lem}\label{lem:character}
 Let $E$ and $E'$ be two elliptic curves defined over $\mathbb{Q}_p$, isomorphic over a quadratic extension $L/\mathbb{Q}_p$. Let $\chi:\mathrm{Gal}(\overline{\mathbb{Q}}_p/\mathbb{Q}_p)\rightarrow \mathrm{Gal}(L/\mathbb{Q}_p)\rightarrow \{\pm 1\}$ be the quadratic character associated to $L/\mathbb{Q}_p$. Then there is an isomorphism $\psi:E\rightarrow E'$ such that $$\psi(P)^{\sigma}=\chi(\sigma)\psi(P^\sigma)\ \text{ for all }\sigma\in \mathrm{Gal}(\overline{\mathbb{Q}}_p/\mathbb{Q}_p) \text{ and all }P\in E(\overline{\mathbb{Q}}_p).$$ 
\end{lem}

Denote by  $\rho_{E, \ell}$ (resp.~$\rho_{E', \ell}$) the Galois representation attached to the $\ell$-torsion of $E$ (resp.~of $E'$). We have that $\rho_{E', \ell}=\chi\otimes \rho_{E, \ell}$. Therefore, if the image of $\rho_{E, \ell}$ is trivial or cyclic, the image of $\rho_{E', \ell}$ is also trivial or cyclic.

\subsection{$2$-torsion of $E_q$}\label{subsec:2-torsion}

Let $p$ be an odd prime number, let $q\in \mathbb{Q}_p$ with $\vert q \vert_p<1$, and $E_q/\mathbb{Q}_p$ the corresponding Tate curve, defined by the equation $y^2 + xy=x^3 + a_4(q) x + a_6(q)$.

For $i=1, 2, 3$, denote by $P_i=(x_i(q), y_i(q))$ its three nontrivial points of $2$-torsion, where $x_i(q), y_i(q)\in \overline{\mathbb{Q}}_p^{\times}$. 

To compute $E_q[2]$, we will use the isomorphism \eqref{eq:Tate-isomorphism}. First, we study the $2$-torsion of $\overline{\mathbb{Q}}_p^{\times}/q^{\mathbb{Z}}$. If $u\in \overline{\mathbb{Q}}_p^{\times}$ satisfies that $[u]^2=[1]$, then there exists $n\in \mathbb{Z}$ such that $u^2=q^n$. Fix a square root $q^{1/2}$ of $q$. Then we have the four elements $u=1, -1, q^{1/2}, -q^{1/2}$ satisfying that $[u]^2=[1]$. It can be easily checked that, since $\vert q\vert_p < 1$, the equivalence classes of these four elements in  $\overline{\mathbb{Q}}_p^{\times}/q^{\mathbb{Z}}$ are all different. We know that $E_q[2]$ has only four elements of $2$-torsion, thus the same must be true of $W_2$. Thus $W_2=\{[1], [-1], [q^{1/2}], [-q^{1/2}]\}$. By definition $\phi([1])=\mathcal{O}$. Relabelling the $2$-torsion points, we have that $\phi([-1])=P_1$, $\phi([q^{1/2}])=P_2$ and $\phi([-q^{1/2}])=P_3$. We can compute the first few terms of $x_i$, $i=1, 2, 3$, by using Equation \eqref{eq:series}. 

More precisely, we have that 
\begin{equation*}
 X(u, q)= \frac{u}{(1-u)^2} + \sum_{n\geq 1}\left(\frac{q^nu}{(1 - q^nu)^2} + \frac{q^n u^{-1}}{(1-q^nu^{-1})^2} - 2\frac{q^n}{(1- q^n)^2}\right)
\end{equation*}

Therefore 

\begin{equation*}\begin{aligned}
 x_1(q)=X(-1, q)&= \frac{-1}{4} + \sum_{n\geq 1}\left(\frac{-q^n}{(1 + q^n)^2} + \frac{-q^n}{(1+q^n)^2} - 2\frac{q^n}{(1- q^n)^2}\right)\\
         &=\frac{-1}{4} -2 \sum_{n\geq 1}\left(\frac{q^n}{(1+q^n)^2} + \frac{q^n}{(1- q^n)^2}\right)\\
         &=\frac{-1}{4} - 4\sum_{n\geq 1} \frac{1 + q^{2n}}{(1 - q^{2n})^2} q^n\\
         &=\frac{-1}{4}  - 4\frac{1 + q^2}{(1-q^2)^2} q  -\cdots\\
         &=\frac{-1}{4}  - 4 q + \mathcal{O}(q^2) 
\end{aligned}\end{equation*}

\begin{equation*}\begin{aligned}
 x_2(q)=X(q^{1/2}, q)&= \frac{q^{1/2}}{(1-q^{1/2})^2} + \sum_{n\geq 1}\left(\frac{q^{n + 1/2}}{(1 - q^{n+1/2})^2} + \frac{q^{n-1/2}}{(1-q^{n-1/2})^2} - 2\frac{q^n}{(1- q^n)^2}\right)\\
                  &= \frac{q^{1/2}}{(1-q^{1/2})^2} + \frac{q^{1/2}}{(1-q^{1/2})^2} - 2\frac{q}{(1- q)^2} + q^{3/2} \frac{1}{(1-q^{3/2})^2} +  q^{3/2} \frac{1}{(1-q^{3/2})^2}+\cdots\\
                  &= 2\frac{q^{1/2}}{(1-q^{1/2})^2}- 2\frac{q}{(1- q)^2} + 2  q^{3/2} \frac{1}{(1-q^{3/2})^2} + \cdots\\
                  &= 2q^{1/2} + 2q + 8 q^{3/2} + \mathcal{O}(q^2)
\end{aligned}\end{equation*}

\begin{equation*}\begin{aligned}
 x_3(q)=X(-q^{1/2}, q)&=  \frac{-q^{1/2}}{(1+q^{1/2})^2} + \sum_{n\geq 1}\left(\frac{-q^{n+1/2}}{(1 + q^{n+1/2})^2} + \frac{-q^{n-1/2}}{(1+q^{n-1/2})^2} - 2\frac{q^n}{(1- q^n)^2}\right)\\
                   &= \frac{-q^{1/2}}{(1+q^{1/2})^2} +\frac{-q^{1/2}}{(1+q^{1/2})^2} - 2\frac{q}{(1- q)^2} - q^{3/2} \frac{1}{(1+q^{3/2})^2} -  q^{3/2} \frac{1}{(1+q^{3/2})^2} +\cdots\\
                   &= 2\frac{-q^{1/2}}{(1+q^{1/2})^2} - 2\frac{q}{(1- q)^2} -  2q^{3/2} \frac{1}{(1+q^{3/2})^2}  + \cdots\\
                                     &= -2q^{1/2} + 2q - 8 q^{3/2} + \mathcal{O}(q^2)
\end{aligned}\end{equation*}

In the three equations above, the last line shows the terms with $p$-adic valuation smaller than or equal to $3/2$ times the valuation of $q$. Note that, if $q^{1/2}\in \mathbb{Z}_p$, then $x_1(q), x_2(q), x_3(q)$ belong to $\mathbb{Z}_p$.

\section{A change of variables}\label{sec:change_of_variables}

In this section we will write out an admissible change of variables, which transforms the Weierstrass equation defining a Tate curve into another, simpler Weierstrass equation, with the property that the difference of the $x$-coordinates of two pairs of non-trivial $2$-torsion points are congruent to $1$ modulo $p$.

Let $E_q/\mathbb{Q}_p$ be a Tate curve, defined by the equation $y^2 + xy=x^3 + a_4(q) x + a_6(q)$. Consider the change of variables
\begin{equation*}\label{eq:change_of_variables}
\begin{cases} x= \frac{1}{36}(x'-3)\\
 y=\frac{1}{2}(\frac{y'}{108}-\frac{x'-3}{36}) 
 \end{cases}
\end{equation*}

We obtain the new Weierstrass equation
\begin{equation*}
y^2=x^3  + (1296a_4(q) - 27)x + 46656a_6(q) - 3888a_4(q) + 54.
\end{equation*} 

The curve defined by this Weierstrass equation, say $E'_q$ is isomorphic (over $\mathbb{Q}_p$) to $E_q$. Let $(x_0, y_0)\in E_q(\overline{\mathbb{Q}}_p)$. Then we have that the point $(x'_0, y'_0)$ such that 
\begin{equation*}
\begin{cases} x_0= \frac{1}{36}(x'_0-3)\\
 y_0=\frac{1}{2}(\frac{y'_0}{108}-\frac{x'_0-3}{36}) 
 \end{cases}
\end{equation*}
belongs to $E'_q(\overline{\mathbb{Q}}_p)$. Therefore we can write $x'_0, y'_0$ in terms of $x_0, y_0$ as 

\begin{equation}\label{eq:newcoordinates}
\begin{cases} x'_0=36 x_0 + 3\\
 y'_0=108(2y_0 + x_0)
 \end{cases}
\end{equation}

In Section \ref{subsec:2-torsion} we computed the $x$-coordinate of the non-trivial $2$-torsion points of $E_q$, namely $x_1(q)$, $x_2(q)$ and $x_3(q)$. Applying the change of variables \eqref{eq:newcoordinates} we obtain the $x$-coordinates of the nontrivial $2$-torsion points of $E'_q$, which are the roots of the polynomial $x^3  + (1296a_4(q) - 27)x + 46656a_6(q) - 3888a_4(q) + 54$. 

Namely, we obtain that 
\begin{equation*}\begin{aligned}
 x'_1(q)&= 36 x_1(q) + 3=-6 - 144 q +\mathcal{O}(q^2) \cdots\\
 x'_2(q)&=36 x_2(q) + 3
                  = 3 + 72q^{1/2} + 72q + 288 q^{3/2} + \mathcal{O}(q^2) \\
 x_3(q)&= 36 x_3(q) + 3
                   = 3 - 72q^{1/2} + 72 q - 288 q^{3/2} +   \mathcal{O}(q^2)\\
                   \end{aligned}
\end{equation*}

Moreover, recall that if $q^{1/2}\in \mathbb{Z}_p$, then $x_1(q)$, $x_2(q)$ and $x_3(q)$ belong to $\mathbb{Z}_p$. 

In that case, we can write $x'_2(q) - x'_1(q)=  9 + 72q^{1/2}\alpha$, $x'_3(q) - x'_1(q)= 9 + 72q^{1/2}\beta$, $x'_2(q) - x'_3(q)=144q^{1/2}\gamma$, where $\alpha, \beta, \gamma\in \mathbb{Z}_p$ are $p$-adic integers congruent to $1$ modulo $p$.

We want that at least two of the differences are congruent to $1$ modulo $p$. Therefore, we make the further change of variables 
\begin{equation*}\label{eq:change_of_variables3}
\begin{cases} x'= 9x''\\
 y'=27y''\\ 
 \end{cases}
\end{equation*}

The new equation that we obtain is
\begin{equation}\label{eq:E2} y^2= x^3  + (16 a_4(q) - \frac{1}{3})x + 64a_6(q) - \frac{16}{3}a_4(q) + \frac{2}{27}\end{equation}
This is an elliptic curve, say $E''_q$, defined over $\mathbb{Q}_p$. If $(x'_0, y'_0)\in E'_q(\overline{\mathbb{Q}}_p)$, then we have that $(x''_0, y''_0)=(x'_0/9, y'_0/27)\in  E''_q(\overline{\mathbb{Q}}_p)$. The $x$-coordinates of the non-trivial $2$-torsion points of this elliptic curve are precisely $x'_1(q)/9$, $x'_2(q)/9$ and $x'_3(q)/9$.

\begin{equation*}\begin{aligned}
 x''_1(q)&= \frac{1}{9}(36 x_1(q) + 3)=-2/3  - 16q  +\mathcal{O}(q^2) 
\end{aligned}\end{equation*}

\begin{equation*}\begin{aligned}
 x''_2(q)&= \frac{1}{9}(36 x_2(q) + 3)= \frac{1}{3} + 4x_2(q) = \\
&= \frac{1}{3} + 8 q^{1/2} + 8q  + 32 q^{2/3} + \mathcal{O}(q^2)\\
\end{aligned}\end{equation*}

\begin{equation*}\begin{aligned}
 x''_3(q)&=\frac{1}{9}(36x_3(q) + 3)=   \frac{1}{3} + 4x_3(q) =\\ 
 &=\frac{1}{3} - 8 q^{1/2} + 8q  - 32 q^{2/3} + \mathcal{O}(q^2)\\
\end{aligned}\end{equation*}

Observe now that 
\begin{equation}\label{eq:x2}\begin{cases}
   x''_2(q) - x''_1(q)= 1 + 8\alpha q^{1/2}\\
   x''_3(q)-x''_1(q)= 1 - 8\beta q^{1/2}\\
   x''_2(q) - x''_3(q)= 64\gamma q^{1/2}\\
  \end{cases}
\end{equation}

where $\alpha, \beta\gamma\in \mathbb{Z}_p$ are $p$-adic integers congruent to $1\pmod{p}$.

\section{Construction of the elliptic curve}

In this section we will show the existence of infinitely many elliptic curves satisfying that their $p$-division fields are locally cyclic extensions of $\mathbb{Q}$. We will proceed in three steps. First, we will write out explicitly a condition modulo a suitable power of $p$, ensuring that the decomposition group at $p$ is cyclic. Then we will address the prime $p=2$; since we have chosen a Weierstrass equation whose discriminant is even, it is necessary to ensure, by means of explicit congruence conditions, that at this prime the elliptic curve has good reduction. Finally, we apply a result of Green and Tao on the representability of primes by linear forms to prove the existence of the desired elliptic curves.

\subsection{Conditions on the Weierstrass equation modulo a power of $p$}\label{sec:Conditions_at_p}

From now on, we fix a prime $p$.
Choose $q=p^{2p}\in \mathbb{Q}_p$; 
one then has $q^{1/2}\in \mathbb{Z}_p$, whence the curve $E_q$ satisfies that $x_1(q), x_2(q), x_3(q)\in \mathbb{Z}_p$ (with the notation of Section \ref{subsec:2-torsion}). Furthermore $q^{1/p}\in \mathbb{Z}_p$, whence the image by $\rho_{E_q, p}$ of the decomposition group at $p$ is a cyclic group by Corollary \ref{cor:cyclic_at_p}. 
%the curve $E_q$ satisfies that $q^{1/2}\in \mathbb{Z}_{p}$ and $q^{1/p}\in \mathbb{Z}_{p}$. Hence, we have that $x_1(q), x_2(q), x_3(q)\in \mathbb{Z}_p$ and furthermore the image by $\rho_{E_q, p}$ of the decomposition group at $p$ is a cyclic group by Corollary \ref{cor:cyclic_at_p}.
Let $a_4(p^{2p})$, $a_6(p^{2p})\in \mathbb{Q}_p$ be given by Equations \eqref{eq:Tate}, let $E''/\mathbb{Q}_p$ be the elliptic curve defined by equation \eqref{eq:E2}, and let $x''_1, x''_2, x''_3\in \mathbb{Z}_p$ be the $x$-coordinates of the non-trival $2$-torsion points of $E''$.

Note that the $p$-adic expansions of $x''_1, x''_2, x''_3$ are the following:

\begin{equation}\label{eq:desarrollos}
 \begin{cases}
   x''_1=-2/3 - 16 p^{2p} + \mathcal{O}(p^{4p})\\
  x''_2=\frac{1}{3} + 8p^{p} + \mathcal{O}(p^{2p})\\
  x''_3= \frac{1}{3} -  8 p^{p} + \mathcal{O}(p^{2p})\\
 \end{cases}
\end{equation}

We want to determine an exponent $r\in \mathbb{N}$ such that $r>p$ and such that, if $a, b, c\in \mathbb{Z}$ satisfy that  $v_p(x''_1 - a)\geq r$, $v_p(x''_2 - b)\geq r$, $v_p(x''_3 - c)\geq r$ then the elliptic curve $E$ defined by the Weierstrass equation $y^2=(x-a)(x-b)(x-c)$ satisfies that the image by the Galois representation $\rho_{E, p}$ of the decomposition group at $p$ is cyclic. We will use the following result on the lifting of roots in $\mathbb{Z}_p$ (cf.~\cite[Theorem 2.24]{NZM}):

\begin{thm}
 Let $f(x)\in \mathbb{Z}[x]$ be a polynomial, $\alpha\in \mathbb{Z}$. Let $\tau$ be the greatest exponent such that $p^{\tau}\vert f'(\alpha)$. Suppose 
that $f(\alpha) \equiv 0\pmod{p^j}$ with  $j \geq 2 \tau + 1$. Then there is a unique $\beta\in \mathbb{Z}_p$ such that $f(\beta)=0$ and $\beta\equiv \alpha\pmod{p^{j-\tau}}$.
\end{thm}

Let $q_0\in \mathbb{Z}$ be such that $q_0\equiv p^{2p} \pmod{p^{4p+1}}$, and let $f(x)=x^{2p} - q_0\in \mathbb{Z}[x]$.
We have that $p$ satisfies the congruence $f(p)=p^{2p} - q_0 \equiv 0\pmod{p^{4p+1}}$, and $f'(p)=2p p^{2p-1}=2p^{2p}$. Therefore we can apply the result with $\tau=2p$ and $j=4p+1$, and we obtain that there exists some $\beta\in \mathbb{Z}_p$ such that $f(\beta)=0$ and $\beta\equiv p\pmod{p^{2p+1}}$. If we consider the elliptic curve $E_{q_0}/\mathbb{Q}_p$, we have that $q_0^{1/2}=\pm\beta^p\in \mathbb{Z}_p$ and $q_0^{1/p}=\zeta \beta^2\in \mathbb{Z}_p$ (for some $p$-th root of unity $\zeta$). In particular, the image by $\rho_{E_{q_0}, p}$ of the decomposition group at $p$ is a cyclic group by Corollary \ref{cor:cyclic_at_p}.

Our problem is reduced to proving that, if $a, b, c$ are $p$-adically close to $x''_1, x''_2$ and $x''_3$, then the resulting elliptic curve will be isomorphic to $E_{q_0}$ with $q_0$ $p$-adically close to $p^{2p}$.

Let $a, b, c\in \mathbb{Z}$ be such that  $v_p(x''_1 - a)$, $v_p(x''_2 - b)$, $v_p(x''_3 - c)\geq 4p+1$, and let $E$ be the elliptic curve defined by the Weierstrass equation $y^2=(x-a)(x-b)(x-c)$. The discriminant  of this Weierstrass equation is
\begin{equation*}
 \Delta_E=16(a-b)^2(b-c)^2(c-a)^2,
\end{equation*}
which has positive $p$-adic valuation because of \eqref{eq:desarrollos}. Moreover, the $c_4$-invariant of this equation is 
\begin{equation*}
 c_{4, E}=(a-b)^2 - (b-c)(a-c),
\end{equation*}
which is a $p$-adic unit. Therefore, the $j$-invariant has negative $p$-adic valuation. Lemma \ref{lem:isomorphic} implies that there exists some $q_0$ with $\vert q_0\vert<1$ such that $E$ is isomorphic to $E_{q_0}$ over an extension of $\mathbb{Q}_p$ of degree at most $2$, and Lemma \ref{lem:character} implies that if the image by $\rho_{E_{q_0}, p}$ of the decomposition group at $p$ is a cyclic group, then the image by $\rho_{E, p}$ is also cyclic. 

In fact, if $j(E)$ is the $j$-invariant of $E$, we can compute $q_0$ by evaluating a power series $g(z)=z + \cdots \in \mathbb{Z}[[z]]$ at $1/j(E)$. Denote by $\Delta_q$ and $c_{4, q}$ the discriminant and $c_4$-invariant of \eqref{eq:E2} for our choice of $q=p^{2p}$.

Since $v_p(x''_1 - a)$, $v_p(x''_2 - b)$, $v_p(x''_3 - c)\geq 4p+1$, it also holds that $\Delta_E \equiv \Delta_q\pmod{p^{4p+1}}$ and $c_{4, E} \equiv c_{4, q} \pmod{p^{4p+1}}$. Since $c_{4, E}$ and $c_{4, q}$ are invertible modulo $p^{4p+1}$, we also have that $c_{4, E}^{-1} \equiv c_{4, q}^{-1} \pmod{p^{4p+1}}$, and therefore 
$1/j_{E}\equiv 1/j_{E_q}\pmod{p^{4p+1}}$.

Therefore $q=g(1/j_{E_q})\equiv g(1/j_{E})=q_0\pmod{p^{4p+1}}$, which is precisely the condition that we needed in order to ensure that the image of $\rho_{E_{q_0},p }$ of the decomposition group at $p$ is cyclic.

\subsection{Conditions at the prime $2$}

Consider the elliptic curve $E_2/\mathbb{Q}$ defined by the Weierstrass equation
$$y^2=x(x-1)(x-17).$$

This elliptic curve has conductor 17 and $j$-invariant $20346417/289$ (labelled 17.a2 in the L-functions and modular forms database \cite{LMFDB}). In particular, it has good reduction at $2$. Now we consider a $2$-adic approximation of this equation. Namely,
let us consider the elliptic curve $E(t_1, t_2, t_3)$ defined by the Weierstrass equation
$$y^2=(x- 64t_1)(x- (1 + 64t_2))(x-(17 + 64t_3))$$
for any $t_1, t_2, t_3\in \mathbb{Z}$. 

\begin{lem}\label{lem:conditions_at_2}
 The elliptic curve $E(t_1, t_2, t_3)$ has good reduction at the prime $2$.
\end{lem}

\begin{proof}
 Consider the change of variables
 \begin{equation*}
  \begin{cases}
   x=4x' +1\\
   y=8y' + 4x'
  \end{cases}
 \end{equation*}

The new Weierstrass equation that we obtain is
\begin{multline*}
 xy + y^2 = x^3 - (16t_1 + 16t_2 + 16t_3 + 4)x^2 - \\ (-256t_1t_2 - 256t_1t_3 - 256t_2t_3 - 64t_1 - 60t_2 + 4t_3 + 1)x - 4096t_1t_2t_3 - 1024t_1t_2 + 64t_2t_3 + 16t_2
\end{multline*}

One can check that the discriminant of this elliptic curve is odd, for any value of $t_1, t_2, t_3\in \mathbb{Z}$. Therefore, the elliptic curve $E(t_1, t_2, t_3)/\mathbb{Q}$ has good reduction at $2$ for any tuple $(t_1, t_2, t_3)\in \mathbb{Z}^3$.

\end{proof}

\subsection{Choosing the elliptic curve}\label{sec:Choosing}

Let $K$ be a number field, and let $E$ be the plane projective curve defined by the (affine) Weierstrass equation

\begin{equation}\label{eq:ce}
 y^2=(x-a)(x-b)(x-c),
\end{equation} for $a, b, c\in K$.

The discriminant of this equation is $\Delta= 16 (a-b)^2(b - c)^2 (c-a)^2$. Thus, if $a, b, c\in K$ are three different elements, the curve $E$ is an elliptic curve defined over the field $K$. Furthermore, if $a, b, c$ are algebraic integers, it has good reduction outside the set of primes dividing $2(a-b)(b-c)(c-a)$.

Our aim is to find  $a_0, b_0, c_0\in \mathbb{Z}$ such that the following congruences hold:
\begin{equation}\label{eq:conditions}
 \begin{cases}
 a_0\equiv 0\pmod{64}\\
  a_0\equiv x''_1 \pmod{p^{4p+1}}\\
 \end{cases}
\begin{cases}
 b_0\equiv 1\pmod{64}\\
  b_0\equiv x''_2 \pmod{p^{4p+1}}\\
 \end{cases}
 \begin{cases}
 c_0\equiv 17\pmod{64}\\
  c_0\equiv x''_3 \pmod{p^{4p+1}}\\
 \end{cases}
\end{equation}
Moreover, we want  that $b_0-a_0$, $c_0-a_0$ are prime numbers, congruent to $1$ mod $p$, and $c_0-b_0$ is equal to $16p^p$ times a prime number which is congruent to $1$ mod $p$.
Let us show how to find such values of $a_0, b_0, c_0\in \mathbb{Z}$. To simplify the formulas, call $r=4p+1$.

First of all, using the Chinese Remainder Theorem we choose auxiliary values $s_1, s_2\in \mathbb{Z}$ such that:

\begin{equation*}
 \begin{cases}
  s_1\equiv 1 \pmod{64}\\
  s_1\equiv x''_2-x''_1\pmod{p^{r}}
 \end{cases}
 \begin{cases}
  s_2\equiv 17 \pmod{64}\\
  s_2\equiv x''_3-x''_1\pmod{p^{r}}
 \end{cases}
\end{equation*}

Consider the following affine linear forms:

\begin{equation*}
 \begin{cases}
  Q_1(X, Y)= s_1 + 64p^rX\\
  Q_2(X, Y)= s_2 + 64 p^rY\\
  Q_3(X, Y)=\frac{s_1-s_2}{16p^p} + 4p^{r-p}(X-Y) \\
 \end{cases}
\end{equation*}

Note that, because of the choice of $s_1$ and $s_2$, the quantity $(s_1-s_2)/16p^p$ is an integer number.

 These forms are pairwise affine-linearly independent, thus we can apply a result of Green and Tao (\cite[Corollary 1.9]{GreenTao}\footnote{Corollary 1.9, as stated in \cite{GreenTao}, is conditional on two conjectures, labeled MS(s), GI(s). However, MN(s) has been proven by Green and Tao in 2008 ["The Möbius function is strongly orthogonal to nilsequences'', preprint, arXiv:0807.1736], and conjecture GI(s) was recently proved by Green, Tao and Ziegler ["An inverse theorem for the Gowers Us+1[N]-norm'', preprint, arXiv:1009.3998]; thus the result is unconditional.}; cf.~\cite[Theorem 3.5]{KimKoenig}).
 
 Hence, we can conclude that there exist infinitely many $x_0$, $y_0$, such that $Q_1(x_0, y_0)$, $Q_2(x_0, y_0)$, $Q_3(x_0, y_0)$ are prime numbers.
 
Let us pick such a pair $(x_0, y_0)$, and call $q_1=Q_1(x_0, y_0)$, $q_2=Q_2(x_0, y_0)$, $q_3=Q_3(x_0, y_0)$ the corresponding primes. 

Choose an integer $a_0\equiv x''_1\pmod{p^r}$ and congruent to $0$ mod 64, and let $b_0=a_0 + q_1$, $c_0= a_0 + q_2$.
Then it holds that 
$$b_0-c_0=q_1-q_2=s_1-s_2 + 64p^r(x_0-y_0)= 16p^p q_3$$ and thus

\begin{equation*}
\begin{cases} b_0-a_0=q_1, \\
 c_0-a_0= q_2\\
 b_0-c_0=16p^pq_3\end{cases}
\end{equation*}

Moreover, since $q_1=b_0-a_0$ is congruent to $x''_2-x''_1\pmod{p^r}$, it holds that $q_1$ is congruent to $1 \pmod{p}$. Similarly, $q_2=c_0-a_0\equiv x''_3-x''_1\pmod{p^r}$ is congruent to $1\pmod{p}$. Finally, 
$q_3=\frac{b_0-c_0}{16p^{p}}$. But
$b_0-c_0\equiv x''_2-x''_3\pmod{p^r}$, and $x''_2-x''_3= 16p^{p} + \mathcal{O}(p^{2p})$. Since $r>p$, it follows that $q_3\equiv 1 \pmod p$.
%and therefore the set of primes dividing $(a-b)(b-c)(c-a)$ is $\{2, p, q_1, q_2, q_3\}$. Furthermore, $A_0, B_0$ and $C_0$ are primes which are congruent to $1$ modulo $p$ by construction.

\begin{cor}
\label{cor:exist}
Assume $p\geq 5$.
 The curve $E/\mathbb{Q}$ defined by the Weierstrass equation \begin{equation}\label{eq:cefinal} y^2=(x-a_0)(x-b_0)(x-c_0)\end{equation}
 is an elliptic curve with good reduction outside of $p, q_1, q_2, q_3$, and multiplicative reduction at $p$, $q_1$, $q_2$, $q_3$.
 
 Denote by $\rho_{E, p}$ the Galois representation attached to the $p$-torsion points of $E$. Then the image by $\rho_{E, p}$ is $\GL_2(p)$, and the images of the decomposition groups at the primes $p, q_1, q_2, q_3$ are cyclic.
\end{cor}

\begin{proof} The discriminant of the Weierstrass equation \eqref{eq:cefinal} is 
 $\Delta= 16 (a_0-b_0)^2(b_0 - c_0)^2 (c_0-a_0)^2 = 16 q_1^2 q_2^2 (16p^pq_3)^2= 2^{12} p^{2p}q_1^2 q_2^2 q_3^2$.
 By construction, the three values $q_1$, $q_2$ and $16p^pq_3$ are nonzero, which ensures that \eqref{eq:cefinal} defines an elliptic curve, with good reduction outside the set of primes dividing the discriminant.
 
 Furthermore, at the primes $q_1, q_2, q_3$, the valuation of the discriminant is less than $12$, and hence the equation \eqref{eq:cefinal} is a minimal Weierstrass equation for $E$.
 To see that the reduction is multiplicative, we need to ensure that the invariant $c_4$ of the Weierstrass equation \eqref{eq:cefinal} is not divisible by these primes. An easy computation yields that 
 $$c_4=(a_0-b_0)^2 - (b_0-c_0)(a_0-c_0)=q_1^2 - 16p^pq_2q_3;$$
 it is clear that this number is not divisible by $q_1, q_2, q_3$.
 
Moreover, for each $i=1, 2, 3$, the $j$-invariant of $E$ has negative $q_i$-adic valuation. By Lemma \ref{lem:isomorphic}, $E$ is isomorphic to $E_{\beta_i}$ for a certain value $\beta_i\in \mathbb{Q}_{q_i}^{\times}$ with positive valuation.

 Since $q_i$ is congruent to $1$ mod $p$, Corollary \ref{cor:1modp} implies that the image of the decomposition group at $q_i$ by $\rho_{E_{\beta_i}, p}$ is cyclic, for $i=1, 2, 3$. Lemma \ref{lem:character} implies that the image by $\rho_{E, p}$ is also cyclic. 
 
 Let us look at the reduction mod $2$. Since $a_0$ is congruent to $0$ mod $64$, we can write it as $a_0=64 t_1$. Similarly $b_0=a_0 + q_1=a_0 + s_1 + 64p^rx_0= 1 + 64 t_2$ and $c_0=a_0 + q_2=a_0 + s_2 + 64 p^r y_0=17 + 64 t_3$ for certain $t_1, t_2, t_3\in \mathbb{Z}$. Therefore, Lemma \ref{lem:conditions_at_2} implies that $E/\mathbb{Q}$ has good reduction at $2$.
 
 Now we study the reduction mod $p$. We perform the following two changes of variables
 
 \begin{equation*}
  \begin{cases}
   x=x'/9\\
   y=y'/27\\
  \end{cases} \hskip 2cm 
  \begin{cases} x'=36x'' + 3\\ y'= 108(2y'' + x'')\\
  \end{cases}
 \end{equation*}
The Weierstrass equation defining $E$ is transformed into a new equation
\begin{multline*}
y^2 + xy= x^3 - \frac{1}{4}(a_0 + b_0 + c_0)x^2 - (-\frac{1}{16}(a_0b_0 + a_0c_0 + b_0c_0) + 1/24(a_0 +b_0 + c_0) - \frac{1}{48})x \\- \frac{1}{64}a_0b_0c_0 + \frac{1}{192}(a_0b_0 + a_0c_0 + b_0c_0) - \frac{1}{576}(a_0 + b_0 + c_0) + \frac{1}{1728}.
\end{multline*}

If we compute the discriminant $\Delta$ of this equation, taking into account that $a_0\equiv x''_1\equiv -2/3 + \mathcal{O}(p^{p+1})\pmod{p^r}$, $b_0\equiv x''_2= 1/3 + 8p^{p} + \mathcal{O}(p^{p+1})\pmod{p^r}$ and $c_0\equiv x''_3= 1/3 - 8p^{p} + \mathcal{O}(p^{p+1})\pmod{p^r}$, we obtain that $p^{p}\vert \Delta$.  However, the invariant $c_4$ is congruent to $1$ mod $p^{p}$. Hence, this equation is a minimal model at $p$, and the curve has multiplicative reduction at $p$. 

Because of the choice of $r$, we have that the image by $\rho_{E, p}$ of the decomposition group at $p$ is cyclic (cf.~Section \ref{sec:Conditions_at_p}). 
 
Finally, note that $E/\mathbb{Q}$ is a semistable elliptic curve. %When $p\geq 11$, a result of Mazur (cf.~\cite[Theorem 4]{Mazur78}) implies that the image of $\rho_{E, p}$ is $\mathrm{GL}_2(\mathbb{F}_p)$.
%It remains to consider the cases $p=5$ and $p=7$. 
If $\rho_{E, p}$ is not surjective, then either $E$ or an elliptic curve $p$-isogenous to $E$, should have a rational point of order $p$ (cf.~\cite{Serre1994-1995}).

 But note that $E$, and with it any curve $p$-isogenous to $E$ over $\mathbb{Q}$, has full $2$-torsion group $\mathbb{Z}/2\mathbb{Z}\times \mathbb{Z}/2\mathbb{Z}$. This yields an elliptic curve over $\mathbb{Q}$ with torsion subgroup containing $\mathbb{Z}/2p\mathbb{Z}\times \mathbb{Z}/2\mathbb{Z}$. By Mazur's torsion theorem (cf.~\cite[Theorem 2]{Mazur78}), and since $p\ge 5$, this is impossible.
 %JK: I think this works for all primes, so distinction p>=11 vs. p<11 not necessary.
 %SA: OK

\end{proof}

Since the $p$-division field of any elliptic curve $E/\mathbb{Q}$ is unramified (and hence, has cyclic decomposition group) at any prime of good reduction other than $p$, we have obtained %the existence part of
Theorem \ref{thm:1}.

%To obtain the linear disjointness part of Theorem \ref{thm:1}, we begin with the following weaker observation.

 %\JK{We conclude with some observations about linear disjointness of the extensions in Theorem \ref{thm:1}. We first consider the extensions obtained in Corollary \ref{cor:exist}. }

\begin{rem} \label{rem:linear_disjointness_over_Qzeta_p}
Assume that $L_1/\mathbb{Q}$ and $L_2/\mathbb{Q}$ are two Galois extensions with Galois group $\GL_2(p)$, obtained as in the proof of Corollary \ref{cor:exist}, such that the prime $q$ has multiplicative reduction for the elliptic curve affording $L_1/\mathbb{Q}$ and is a prime of good reduction for the elliptic curve affording $L_2/\mathbb{Q}$. Then $L_1$ and $L_2$ are linearly disjoint over $\mathbb{Q}(\zeta_p).$

Indeed, the Galois group of $L_1/\mathbb{Q}(\zeta_p)$ and $L_2/\mathbb{Q}(\zeta_p)$ is $\SL_2(p)$. Therefore  $\mathrm{Gal}(L_1/L_1\cap L_2)$ is a normal subgroup of $\SL_2(p)$, hence it is either trivial or a cyclic group of order two, or it is the whole of $\SL_2(p)$. But $L_1\not\subset L_2$, since $q$ is ramified in $L_1/\mathbb{Q}$ and is unramified in $L_2/\mathbb{Q}$. The second possibility cannot hold, because the inertia group at $q$ of $L_1/\mathbb{Q}(\zeta_p)$ is a cyclic group of order $p>2$, which cannot be contained in the extension $L_1/L_1\cap L_2$. 
Therefore we conclude that $L_1\cap L_2=\mathbb{Q}(\zeta_p)$.
\end{rem}

\subsection{Changing solutions via twisting}\label{sec:twisting}
We conclude with a strengthening of the conclusion of Remark \ref{rem:linear_disjointness_over_Qzeta_p}. Concretely, we show, applying a ``twisting" argument, that the locally cyclic extensions obtained above can be changed to be linearly disjoint even over a quadratic number field.  For this, we require a lemma. 

\begin{lem}
\label{lem:twist}
Let $G$ be a finite group, $N\le Z(G)$ a cyclic central subgroup of $G$ and $H:=G/N$. Let $K/\mathbb{Q}$ be a locally cyclic Galois extension with group $H$, and assume that $K$ embeds into at least one locally cyclic Galois extension with group $G$. Then it embeds into infinitely many such extensions, which may additionally be assumed linearly disjoint over $K$.
\end{lem}
\begin{proof}
Pick a locally cyclic $G$-extension $L/\mathbb{Q}$ containing $K$, and let $\mathcal{S}$ be the set of (finite) primes of $\mathbb{Q}$ ramifying in $L/\mathbb{Q}$. Consider the set of all epimorphisms $\varphi: G_{\mathbb{Q}}\to N$ arising from Galois extensions $F/\mathbb{Q}$ with Galois group $N$, such that all finite primes ramifying in $F/\mathbb{Q}$ are completely split in $L$ and {\it totally} tamely ramified in $F/\mathbb{Q}$. There are infinitely many such $\varphi$ (for example, for any of the infinitely many primes $q$ which split completely in $L(\zeta_{|N|})$, the cyclotomic extension $\mathbb{Q}(\zeta_q)/\mathbb{Q}$ has a cyclic subextension of group $N$, ramified only at $q$). Furthermore, since there are only finitely many possibilities for the restriction $\varphi_{|D_q}$ to the decomposition group at $q$ (for each fixed prime $q$), there are also infinitely many ways to choose two such epimorphisms, say $\varphi_1$ and $\varphi_2$, for which the restrictions $(\varphi_1)_{|D_q}$ and $(\varphi_2)_{|D_q}$ coincide for all $q\in \mathcal{S}$. Up to choosing the sets of ramified primes of $\varphi_1$ and $\varphi_2$ distinct, we also obtain that the fixed fields of the kernels of $\varphi_1$ and of $\varphi_2$ are linearly disjoint over $\mathbb{Q}$. Let $\tilde{F}$ be the fixed field of the kernel of the product map $G_{\mathbb{Q}}\to N$, $x\mapsto \varphi_1(x)\cdot \varphi_2^{-1}(x)$. 
Then $\tilde{F}/\mathbb{Q}$ is again Galois of group $N$, all its ramified primes are completely split in $L$, and by construction, all $q\in \mathcal{S}$ are completely split in $\tilde{F}$.
We may now twist $L/\mathbb{Q}$ by $\tilde{F}/\mathbb{Q}$ in the following way: Up to choosing the ramified primes of $\tilde{F}$ sufficiently large, we may assume that $L/\mathbb{Q}$ and $\tilde{F}/\mathbb{Q}$ are linearly disjoint, and hence $\textrm{Gal}(L\tilde{F}/\mathbb{Q}) \cong G\times N$. Let $\tilde{L}$ be the fixed field of the ``diagonal subgroup $D:=\{(n,n)\mid n\ \in N\}$ in this direct product. Then $\tilde{L}/\mathbb{Q}$ is again a $G$-extension, and its finite ramified primes are at most the primes in $\mathcal{S}$ together with the ramified primes of $\tilde{F}/\mathbb{Q}$. Since the latter are totally split in $L$, their decomposition group in $\tilde{L}/\mathbb{Q}$ is the same as in $\tilde{F}/\mathbb{Q}$, i.e., cyclic and embedding into $N$. Since on the other hand, the former primes are completely split in $\tilde{F}/\mathbb{Q}$, their decomposition group is the same as in $L/\mathbb{Q}$, i.e., cyclic by assumption. Linear disjointness over $K$ of any two of the extensions $\tilde{L}/\mathbb{Q}$ obtained in this way is easily achieved via noting that, in the construction of the fields $\tilde{F}$, one may choose the smallest ramified prime arbitarily large (e.g., larger than any prime ramified in any one of finitely many extensions already obtained when repeating the process) and totally ramified.
\end{proof}

We can now deduce the following strengthening of Theorem \ref{thm:1}. %linear disjointness statement of Theorem \ref{thm:1}. Indeed, 
\begin{cor}
\label{thm:disj}
For any prime $p\ge 5$, there are infinitely many locally cyclic Galois extensions of $\mathbb{Q}$ with Galois group $GL_2(p)$ and which are pairwise linearly disjoint over the field $\mathbb{Q}(\sqrt{p^\star}) := \mathbb{Q}(\sqrt{(-1)^{(p-1)/2}p})$.
\end{cor}
\begin{proof}
It suffices to choose $N:=Z(GL_n(p)) \cong \mathbb{Z}/(p-1)\mathbb{Z}$, and recall that, via varying the primes $q_1, q_2, q_3$ in the above proof (cf.~Remark \ref{rem:linear_disjointness_over_Qzeta_p}), we may obtain infinitely many locally cyclic $GL_2(p)$-extensions, pairwise linearly disjoint over $\mathbb{Q}(\zeta_p)$, whence their $PGL_2(p)$-subextensions are pairwise linearly disjoint even over $\mathbb{Q}(\sqrt{p^\star})$. Successively replacing each of those $GL_2(p)$-extensions by a suitable ``twist" as provided by Lemma \ref{lem:twist} then yields the assertion.
\end{proof}
\bibliography{Bibliog}
\bibliographystyle{alpha}

\end{document}